\documentclass[11pt,a4paper,reqno]{amsart}
\usepackage[latin1]{inputenc}
\usepackage[english]{babel}
\usepackage{amsmath}
\usepackage{amsfonts}
\usepackage{amssymb}
\usepackage{mathrsfs}
\usepackage{latexsym}
\usepackage{yfonts}
\usepackage{natbib}

\usepackage{mathrsfs}

\usepackage[margin=2.5cm]{geometry}

\usepackage{color}

\usepackage{graphicx,color}

\usepackage[plainpages=false,colorlinks,hyperindex,bookmarksopen,linkcolor=red,citecolor=blue,urlcolor=blue]{hyperref}

\bibpunct{[}{]}{;}{n}{,}{,}

\newtheorem{theorem}{Theorem}[section]

\newtheorem{lemma}{Lemma}
\newtheorem{corollary}{Corollary}

\newtheorem{remark}{Remark}[section]

\numberwithin{equation}{section}

\allowdisplaybreaks

\author{Mirko D'Ovidio}
\address{Department of Basic and Applied Sciences for Engineering,\newline \indent Sapienza University of Rome,  
\newline \indent via A. Scarpa 10, Rome, Italy}

\email[Corresponding author]{mirko.dovidio@uniroma1.it}

\title{Fractional boundary value problems}

\begin{document}

\maketitle

\begin{abstract}
We study some functionals associated with a process driven by a fractional boundary value problem (FBVP for short). By FBVP we mean a Cauchy problem with boundary condition written in terms of a fractional equation, that is an equation involving time-fractional derivative in the sense of Caputo. We focus on lifetimes and additive functionals characterizing the boundary conditions. We show that the corresponding additive functionals are related to the fractional telegraph equations. Moreover, the fractional order of the derivative gives a unified condition including the elastic and the sticky cases among the others.
\end{abstract}


\section{Introduction}
 \label{sec:1}

\setcounter{section}{1}
\setcounter{equation}{0}\setcounter{theorem}{0}

In this work we focus on the fractional boundary value problem
\begin{equation}
\label{PDEprocessMain}
\begin{cases}
\displaystyle \frac{\partial u}{\partial t} = \Delta u, \quad t>0,\; x \in [0, \infty)\\ 
\displaystyle \eta D^{\alpha/2}_t u(t,0) = \sigma \frac{\partial u}{\partial x} (t, 0) -c\, u(t,0), \quad t>0\\
\displaystyle u(0,x) = f(x), \quad x>0
\end{cases}
\end{equation}
(we denote by $\Delta$ the operator $\partial^2_{xx}$) with $\alpha \in (0,1]$, $\eta \geq 0$ $\sigma\geq 0$, $c\geq 0$ and
\begin{align*}
D^{\alpha/2}_t u(t,x) = \frac{1}{\Gamma(1-\alpha/2)} \int_0^t \frac{\partial u}{\partial s}(s,x)\, (t-s)^{-\alpha/2} ds,
\end{align*} 
that is the Caputo derivative. We are interested in the probabilistic representation of the solution and the associated functionals.  \\

This problem has been inspired by the following observation. Let us consider the Cauchy problem
\begin{align*}
\frac{\partial u}{\partial t}(t,x) = \Delta u(t,x), \quad t\geq 0,\; x \in [0, \infty), \quad u(0,x)=f(x) 
\end{align*}
with the boundary condition
\begin{align*}
\eta D^{\beta}_t u(t,0) = \sigma \frac{\partial u}{\partial x} (t, 0) -c\, u(t,0), \quad t>0
\end{align*}
where $\beta \in (0,1]$. The Caputo derivative enjoys the  continuity property w.r. to the order $\beta$. In particular, as $\beta \uparrow 1$ the Caputo derivative becomes $D^1_t u=\partial u/ \partial t$ and the problem \eqref{PDEprocessMain} can be intuitively regarded as the Cauchy problem associated with the domain
\begin{align}
\label{DomainSticky}
\left\lbrace \varphi \in C^2([0, \infty)) \to \mathbb{R} :  \eta \varphi^{\prime \prime}(0^+) = \sigma \varphi^\prime(0^+) - c \varphi(0^+)  \right\rbrace .
\end{align}
As $\beta \to 1/2$ we have that \eqref{PDEprocessMain} must be associated with the set of functions
\begin{align}
\label{DomainElastic}
\left\lbrace \varphi \in C^2([0, \infty)) \to \mathbb{R} :  -\eta \varphi^{\prime }(0^+) = \sigma \varphi^\prime(0^+) - c \varphi(0^+)  \right\rbrace
\end{align}
where $-\eta \varphi^\prime$ is somehow obtained from the fact that $\eta D^1_t u = \eta \Delta u$ and $\eta D^\beta_t u$ coincides, for $\beta=1/2$, with $-\eta \partial_x u$ (at $x=0$). Thus, the fractional boundary condition involving the Caputo derivative introduces a characterization of the associated process given in terms of an elastic Brownian motion for $\beta \in (0,1/2]$ and in terms of a Sticky Brownian motion for $\beta \in (1/2,1]$. The domains \eqref{DomainSticky} and \eqref{DomainElastic} are often associated with the Wentzell and Robin boundary conditions. The domain \eqref{DomainElastic} suggests also an interesting balance between $\eta$ and $\sigma$ in terms of reflection, especially for negative $\sigma$. Concerning \eqref{DomainSticky} and the problem with $\beta \in (1/2,1]$, the probabilistic representation is still an open problem. The author is working in this direction. \\

In order to give a clear picture about our results we first underline the main aspects relating such results with the well-known theory of time-changed processes. Let us consider the Cauchy problem  
\begin{equation}
\label{CP}
\begin{cases}
\displaystyle \frac{\partial u}{\partial t} = \Delta u\\
\displaystyle u(0,x) = f(x)
\end{cases}
\end{equation} 
for which we have the probabilistic representation of the solution given by $\mathbf{E}_x[f(\widetilde{X}_t)]$. The process $\widetilde{X}$ is a Brownian motion with generator $(\Delta, D(\Delta))$ and the boundary conditions can be associated with a multiplicative functional $M_t = \exp A_t$ where $A_t$ is an additive functional (\cite{BluGet68}). Thus, the solution to \eqref{CP} has the representation
\begin{align*}
\mathbf{E}_x[f(\widetilde{X}_t)]= \mathbf{E}_x[f(X_t) \, M_{t}]
\end{align*} 
where $\widetilde{X}$ is the part process of $X$ in a domain according with $M_t$. For $\alpha \in (0,1]$, the fractional Cauchy problem
\begin{equation}
\label{FCP}
\begin{cases}
\displaystyle D^\alpha_t u = \Delta u\\
\displaystyle u(0,x) = f(x)
\end{cases}
\end{equation} 
has been extensively investigated by many researchers. It is well-known that the time-changed process $\widetilde{X} \circ L$ can be considered in order to solve \eqref{FCP} where $\widetilde{X}$ has generator $(\Delta, D(\Delta))$ and the independent random time $L$ is an inverse to a stable subordinator $H$. The random clock $L$ for the base process $X$ seems to introduce a delaying effect along the paths of $X$. The additive functional turns out to be time-changed as well, in place of $A$ we have to consider $ A\circ L$, that is 
\begin{align}
\label{ALfunctional}
A_{L_t}
\end{align} 
which characterizes the boundary behaviour of the time-changed process. As $\alpha\to 1$ the problem \eqref{FCP} becomes the problem \eqref{CP}. This is due to the fact that $L_t \to t$ almost surely as $\alpha \to 1$. In the special case $\alpha/2 \to 1$, the problem \eqref{PDEprocessMain} takes a completely different face. In the literature the following slightly modified version of \eqref{PDEprocessMain} has been considered (\cite{SW05, SW06})
\begin{equation}
\label{PDEprocessSticky}
\begin{cases}
\displaystyle \frac{\partial u}{\partial t} = \mu \nabla u + \Delta u\\
\displaystyle \frac{\partial u}{\partial t}(t,0) = \sigma \frac{\partial u}{\partial x}(t,0), \quad t>0\\
\displaystyle u(0,x)= f(x)
\end{cases}
\end{equation}
with $\mu, \sigma \in \mathbb{R}$ (we denote by $\nabla u$ the first derivative of $u$). The ideal interpretation of \eqref{PDEprocessSticky} can be given by considering the set of functions
\begin{align}
\label{anlogyDomain}
\left\lbrace \varphi \in C^2([0, \infty)) \to \mathbb{R} :  \mu \varphi^\prime(0^+) + \varphi^{\prime \prime}(0^+) = \sigma \varphi^\prime(0^+)  \right\rbrace
\end{align}
from which the sticky condition immediately emerges. For $\mu=0$, the sticky condition introduces the time change
\begin{align}
\label{RandomTimeV}
V_t = t + \frac{1}{\sigma} \gamma_t
\end{align}
where $\gamma_t$ is a local time at zero. In particular, a reflecting Brownian motion time-changed by the inverse $V^{-1}_t$ can be considered in order to solve the heat equation with sticky condition at zero.\\

In the present work, the first result we provide is that the formula \eqref{ALfunctional} does not hold. Indeed, the functional associated with the fractional boundary condition in \eqref{PDEprocessMain} is a composition as expected but it turns out to be 
\begin{align}
\label{TAfunctional}
\bar{L}_{A_t}
\end{align}
where $\bar{L}$ is an inverse process playing the role of $L$ in \eqref{ALfunctional}. Moreover, we show that $\bar{L}$ is an inverse to the process
\begin{align*}
\bar{H}_t = \frac{\sigma}{\eta} t + H_t
\end{align*} 
which is therefore very close to the probabilistic interpretation of the problem \eqref{PDEprocessSticky}. The process $\bar{L}$ can be associated to a fractional telegraph equation. We recall that $L$ in \eqref{ALfunctional} is an inverse to $H$, thus the composition \eqref{TAfunctional} maintains the same structure of \eqref{ALfunctional}. In particular, if $\sigma=0$, then the random times $\bar{L}$ and $L$ are identical in law. Only in this case, $\sigma=0$, we have that $L \circ A$ equals in law $A \circ L$. In \eqref{PDEprocessMain}, $\alpha/2 \in (0, 1/2]$ and therefore we are not able (up to now) to obtain \eqref{PDEprocessSticky} as a special case of \eqref{PDEprocessMain}. We study the lifetimes of the processes associated with \eqref{PDEprocessMain} and provide a connection with the fractional telegraph equation.\\

Starting from the pioneering works \cite{caputoBook, CapMai71,CapMai71b} the fractional calculus has attracted many researchers working on many different areas of Mathematics and Applied Sciences. A number of fractional or non-local operators have been introduced so far. Although we focus only on the Caputo fractional derivative, many interesting papers have to be considered, the literature is truly extensive. We mention only few references throughout the paper. An interesting survey with a detailed early historical description has been given in \cite{CapMai07}. Many important contributions have been given by Kochubei in a series of papers starting from the work \cite{Koc89} and by Mainardi together with many other collaborators, here we mention the paper \cite{MLP2001}. Some further references on fractional calculus with key results also linking fractional calculus with probability are given by \cite{MNV09,OB09} and successively by \cite{Dov12, GLY15, Kolo19, luchko20} among the others. These works are often focused on the so-called fractional Cauchy problem which has been deeply investigated in \cite{Baz2000, BaeMee2001}. Also in this case, there are many works to be mentioned. Meerschaert (together with many collaborators) provided many contributions to this field (see for example \cite{MeeSikBook} and the references therein). The interested reader can also consult the famous books \cite{SKM93, Kiryakova94, Pod99} and the recent books \cite{BooKolo19,BookKRY}. The strong connection between fractional calculus and probability is getting stronger year by year. From the analysis point of view, some papers deal with fractional boundary value problem meaning the fractional Cauchy problem. Our case is related with the fractional dynamical boundary conditions. An interesting discussion about the physical derivation of such conditions can be found in \cite{Gold2006}. We underline the special role of the "time" boundary condition which may appear useful also in case of irregular boundaries for which the definition of space-operators may be hard to achieve (\cite{CapDov19, CapDov21, CreoLancia20}). As far as we know there are no results on the problem we deal with here.  


\section{Fractional derivatives and random times}
\label{sec:2}

\setcounter{section}{2}
\setcounter{equation}{0}\setcounter{theorem}{0}

Let us consider $b>0$ and the set $AC([0,b])$ of continuous functions with derivative in $L_1([0,b])$. Denote by $v^\prime$ the derivative of $v$ as usual. Thus, $v\in AC([0,b])$ is such that $v^\prime = g \in L_1([0,b])$, that is $v$ has the representation
\begin{align*}
v(t) = v(0) + \int_0^t g(s)ds.
\end{align*} 
We notice that the Sobolev space $W^{1,1}([0,b]) = \{\varphi \in L_1([0,b])\, :\, \varphi^\prime \in L_1([0,b])\}$ endowed with the norm $\|\varphi \|_{W^{1,1}} = \|\varphi\|_{L_1} + \|\varphi^\prime \|_{L_1}$ coincides with the space $AC([0,b])$ for $b<\infty$. Let $v\in AC([0,b])$ and $\alpha \in (0,1)$. For the Riemann-Liouville derivative 
\begin{align}
\mathcal{D}^\alpha_t v(t) :
= & \frac{1}{\Gamma(1-\alpha)} \frac{d}{dt} \int_0^t v(s) (t-s)^{-\alpha} ds \notag \\
= & \frac{1}{\Gamma(1-\alpha)} \left( \frac{v(0)}{t^\alpha} + \int_0^t v^\prime(s) (t-s)^{-\alpha} ds \right) \label{RelationRLC}
\end{align}
we have that \eqref{RelationRLC} exists a. e. on $[0,b]$ and $\mathcal{D}^\alpha_t v \in L_p([0,b])$ with $1\leq p \leq 1/\alpha$ (see \cite[page 28]{Diethelm}). The formula \eqref{RelationRLC} introduces the Caputo derivative
\begin{align}
\label{def:caputoDer}
D^\alpha_t v(t) := \frac{1}{\Gamma(1-\alpha)} \int_0^t v^\prime(s) (t-s)^{-\alpha} ds =  \mathcal{D}^\alpha_t \big( u(t) - u(0) \big). 
\end{align} 
In the literature, the convolution-type operator \eqref{def:caputoDer} is also known as Caputo-Dzherbashian derivative. Indeed the second author actively investigated this operator starting from the papers \cite{Dzh66, DzhNers68}. We immediately see that $\mathcal{D}^\alpha_t$ and $D^\alpha_t$ coincide for compactly supported functions in $C^1_c([0, b))$. We also observe that the derivative $D^\alpha_t v$ is well-defined on the set of functions for which 
$$v^\prime(s) (t-s)^{-\alpha} \in L_1([0, t)) \quad \textrm{for any} \quad t>0.$$ 
In particular, for a function $v$ such that $|v^\prime(s)|\leq s^{\beta-1}$, 
\begin{align*}
|D^\alpha_t v(t)| \leq D^\alpha_t t^\beta = \frac{\Gamma(\beta +1)}{\Gamma(\beta-\alpha+1)} t^{\beta -\alpha}, \quad t>0
\end{align*}
which can be easily obtained from the Beta integral
\begin{align*}
\int_0^1 s^{\beta -1}(1-s)^{-\alpha}ds = \frac{\Gamma(\beta) \Gamma(1-\alpha)}{\Gamma(\beta-\alpha+1)}.
\end{align*}
This fact will be considered below in the definition of the spaces $C^\alpha_L$ and $D^\alpha_L$.

The fractional derivatives above are directly related with stable subordinators and their inverses. For $\alpha \in (0,1)$, let $H$ be the stable subordinator for which
\begin{align}
\label{symbStable}
\mathbf{E}_x[e^{-\lambda H_t}] = e^{-t \lambda^\alpha}, \quad \lambda \geq 0.
\end{align}
The inverse process $L$ is defined as $L_t := \inf\{ s\geq 0\,:\, H_s > t \}$. The subordinator $H$ may have jumps, so that the inverse $L$ may have  plateaux. This gives a clear picture about the well-known behaviours of the time-changed processes obtained by considering $H$ and $L$ as random times. We denote by $h$ and $l$ the densities 
\begin{align*}
\mathbf{P}_0(H_t \in dx) = h(t,x)\, dx \quad \mathbf{P}_0(L_t \in dx) = l(t,x)\, dx.
\end{align*}
After some calculation, by exploiting the equality
\begin{align*}
\mathbf{P}_0(H_s > t) = \mathbf{P}_0(s > L_t),
\end{align*} 
we obtain 
\begin{align}
\label{Lapl}
\int_0^\infty e^{-\lambda t} l(t,x)\, dt = \lambda^{\alpha-1} e^{-x \lambda^\alpha}, \quad \lambda>0.
\end{align} 
We recall the following result which will be useful further on,
\begin{align}
\int_0^\infty e^{-\lambda t}\, h(t,x)\, dt = x^{\alpha-1} E_{\alpha, \alpha} (-\lambda x^\alpha) = - \frac{1}{\lambda} \frac{d}{dx} E_\alpha(-\lambda x^\alpha).
\end{align}
We also recall that the solution to $D^\alpha_t u(t) = -c u(t)$  with $u(0)=1$ is the analytic function $u(t)=E_\alpha(-ct^\alpha)$, that is $u \in C^\infty ((0, \infty))$ is the Mittag-Leffler function
\begin{align}
\label{MLfunction}
E_\alpha(-ct^\alpha) = \sum_{k \geq 0} \frac{(-ct^\alpha)^k}{\Gamma(\alpha k +1)}, \quad t\geq 0, \quad c\geq 0.
\end{align}
As we can see \eqref{MLfunction} is not an element of $L_1((0, \infty))$ but it is in $AC([0, \infty))$.

Let us introduce the space
\begin{align*}
W^{1,1}_0 ([0, \infty)) = \{  \varphi \in L_1([0, \infty)\,:\, \varphi^\prime \in L_1([0, \infty),\, \varphi(0)=0 \}.
\end{align*}
The Riemann-Liouville derivative $-\mathcal{D}^\alpha_x$ gives the generator of the $\alpha$-stable subordinator $H$. In particular, $h \in D(-\mathcal{D}^\alpha_x)$ only if $h \in W^{1,1}_0([0, \infty))$.	 It is well-known that for the inverse process $L$, the density satisfies the fractional equation
\begin{align*}
D^\alpha_t \, l = - \frac{\partial l}{\partial x}.
\end{align*}

For the sake of simplicity we now set $\eta=1$ and we consider the process $\bar{H}_t = \sigma t + H_t$ and its inverse $\bar{L}_t= \inf\{s \geq 0\,:\, \bar{H}_s > t\}$. The drift coefficient $\sigma$ is assumed to be non negative. However, we have interesting applications of our results in case we allow $\sigma \in (-1, 0)$, that is the case of subordinators with negative drift. Let us denote by $\bar{h}$ and $\bar{l}$ the corresponding densities. For the sake of all-inclusive presentation we provide the following results which have been similarly obtained and stated in the literature.

\begin{theorem}
\label{thm:eqRT}
Let us consider
\begin{align*}
\alpha \in (0,1], \quad \sigma \geq 0.
\end{align*}
The solution on $C^{1,1}((0, \infty) \times W^{1,1}_0([0, \infty)), [0, \infty))$ of the problem
\begin{equation}
\label{eq:RTH}
\begin{cases}
\displaystyle \frac{\partial v}{\partial t} = - \sigma \frac{\partial v}{\partial x} - \mathcal{D}^\alpha_x v\\
\displaystyle v(0,x)= f(x), \quad x>0, \quad f \in W^{1,1}_0([0, \infty))
\end{cases}
\end{equation}
is written as
\begin{align*}
v(t,x) = \int_0^x f(x-y)\, \bar{h}(t,y)\, dy = \mathbf{E}[f(x-\bar{H}_t) \mathbf{1}_{(t < \bar{L}_x)}].
\end{align*}
The solution on $C^{1,1}(AC((0, \infty))\times (0, \infty), (0, \infty))$ to the problem
\begin{equation}
\label{eq:RTL}
\begin{cases}
\displaystyle \sigma \frac{\partial w}{\partial t}  + D^\alpha_t w = - \frac{\partial w}{\partial x}\\
\displaystyle w(0,x) = f(x), \quad x>0, \quad f \in C_b(0, \infty)\\
\displaystyle w(t,0)=0, \quad t>0
\end{cases}
\end{equation}
is written as
\begin{align*}
w(t,x) = \int_0^x f(x-y)\, \bar{l}(t,y)\, dy = \mathbf{E}[f(x - \bar{L}_t) \mathbf{1}_{(t < \bar{H}_x)}].
\end{align*}
\end{theorem}
\begin{proof}
Since $v(t, \cdot) \in W^{1,1}_0([0, \infty))$, $t > 0$, we have that
\begin{align*}
\int_0^\infty e^{-\xi x} \mathcal{D}^\alpha_x v(t,x)\, dx = \xi^\alpha\, \widetilde{v}(t, \xi). 
\end{align*}
From the problem \eqref{eq:RTH} we get the double Laplace transform
\begin{align*}
\widetilde{v}(\lambda, \xi) 
\displaystyle = & (\lambda + \sigma \xi + \xi^\alpha)^{-1} \widetilde{f}(\xi)\\
\displaystyle = & \left( \int_0^\infty e^{-\lambda t} e^{-t (\sigma \xi + \xi^\alpha)} dt \right) \widetilde{f}(\xi)\\
\displaystyle = &  \left( \int_0^\infty e^{-\lambda t} \mathbf{E}_0[e^{-\xi \bar{H}_t}] dt \right) \widetilde{f}(\xi)
\end{align*}
where
\begin{align}
\label{lapXI}
\mathbf{E}_0[e^{-\xi \bar{H}_t}]  = \int_0^\infty e^{-\xi x} \bar{h}(t,x)\, dx.
\end{align}
Thus,
\begin{align*}
\mathbf{E}_0[e^{-\xi \bar{H}_t}] \widetilde{f}(\xi) = \int_0^\infty e^{-\xi x} \int_0^x f(x-y) \bar{h}(t,y)\, dy\, dx
\end{align*}
and the result follows. Concerning the probabilistic representation, we only notice that
\begin{align*}
\int_0^x f(x-y) \bar{h}(t,y)\, dy 
= & \int_0^\infty f(x-y) \bar{h}(t,y)\, \mathbf{1}_{(y<x)} dy\\ 
= & \mathbf{E}_0[f(x-\bar{H}_t)\mathbf{1}_{(\bar{H}_t < x)}]
\end{align*}
with $(\bar{H}_t < x) \equiv (t < \bar{L}_x)$ under $\mathbf{P}_0$.\\

From the problem \eqref{eq:RTL} we have that $w(\cdot, x) \in AC((0, \infty))$ for any $x \in (0, \infty)$, thus the derivatives in time are well-defined. In particular, we have the extra-condition $\partial_t w \in C((0, \infty)) \cap L_1((0, \infty))$. We get that
\begin{align*}
\widetilde{w}(\lambda, \xi) 
\displaystyle = & (\sigma \lambda + \lambda^\alpha + \xi)^{-1} \frac{\sigma \lambda + \lambda^\alpha}{\lambda} \widetilde{f}(\xi)\\
\displaystyle = & \left( \int_0^\infty e^{-\xi x} \frac{\sigma \lambda + \lambda^\alpha}{\lambda} e^{-x (\sigma \lambda + \lambda^\alpha)} \right) \widetilde{f}(\xi)
\end{align*} 
where
\begin{align}
\label{lapLAMBDA}
\frac{\sigma \lambda + \lambda^\alpha}{\lambda} e^{-x (\sigma \lambda + \lambda^\alpha)} 
= & - \frac{d}{d x} \frac{1}{\lambda} e^{-x (\sigma \lambda + \lambda^\alpha)}\notag \\
= & - \frac{d}{d x} \frac{1}{\lambda} \mathbf{E}_0[e^{-\lambda \bar{H}_x}]\notag \\
= & - \frac{d}{d x} \int_0^\infty e^{-\lambda t} \mathbf{P}_0(\bar{H}_x < t)\, dt\notag \\
= & - \frac{d}{d x} \int_0^\infty e^{-\lambda t} \mathbf{P}_0(x < \bar{L}_t)\, dt \notag \\
= & \int_0^\infty e^{-\lambda t} \bar{l}(t,x)\, dt.
\end{align}
From this, we write
\begin{align*}
w(t,x) = \int_0^x f(x-y) \bar{l}(t,y)\, dy 
\end{align*}
that is
\begin{align*}
w(t,x) = \mathbf{E}_0[f(x-\bar{L}_t) \mathbf{1}_{(\bar{L}_t < x)}].
\end{align*}
As before, by noticing that $(\bar{L}_t < x) \equiv (t < \bar{H}_x)$ under $\mathbf{P}_0$ we get the result.
\end{proof}

We observe that
\begin{align*}
\bar{h}(t,x)>0, \quad t\geq0, \quad x\in (0, \infty) 
\end{align*}
and
\begin{align*}
\bar{h}(t,x) = 0, \quad t\geq 0,\quad x \in (-\infty, 0]
\end{align*}
whereas
\begin{align*}
\bar{l}(t,x)>0, \quad t\geq 0, \quad  x \in \left[ 0,  \frac{t}{\sigma} \right)
\end{align*}
and
\begin{align*}
\bar{l}(t,x)=0, \quad t\geq 0, \quad x \in (-\infty, 0) \cup \left[ \frac{t}{\sigma}, \infty \right).
\end{align*}
Indeed, by construction,
\begin{align*}
\mathbf{P}_0(\bar{L}_t > t/\sigma) = \mathbf{P}_0(t > \bar{H}_{t/\sigma}) = \mathbf{P}_0(0 > H_{t/\sigma}) = 0.
\end{align*}
Moreover, from \eqref{lapXI} in the previous proof, we have that, $\forall\, t>0$,
\begin{align*}
\xi^n \widetilde{\bar{h}}(t, \xi) \to 0 \quad \textrm{as} \quad \xi \to \infty \quad \forall\, n \in \mathbb{N} \cup \{0\}
\end{align*}
and
\begin{align*}
\xi^n \widetilde{\bar{h}}(t, \xi) \to 0 \quad \textrm{as} \quad \xi \to 0 \quad \forall\, n \in \mathbb{N} \cup \{0\}.
\end{align*}
In particular, $\forall t>0$, $\forall n \in \mathbb{N}$, $\xi^n \widetilde{\bar{h}}(t, \xi) \in C_0((0, \infty))$ and
\begin{align*}
\exists M_H>0\; :\; \forall t > 0,\; | \xi^n \widetilde{\bar{h}}(t, \xi)| \leq M_H.
\end{align*} 
Thus, $\forall t>0$, $\bar{h}(t, \cdot) \in C^\infty ((0, \infty))$. Furthermore, we easily see that $\forall t>0$, $\bar{h}(t, \cdot) \in L_1((0, \infty)$. Concerning the inverse process, from \eqref{lapLAMBDA} we have that, $\forall\, x>0$,
\begin{align*}
\lambda^n \widetilde{\bar{l}}(\lambda, x) \to 0 \quad \textrm{as} \quad \lambda \to \infty \quad \forall\, n \in \mathbb{N}
\end{align*}
and
\begin{align*}
\lambda^n \widetilde{\bar{l}}(\lambda, x) \to 0 \quad \textrm{as} \quad \lambda \to 0 \quad \forall\, n \in \mathbb{N}.
\end{align*}
Thus, 
\begin{align*}
\exists M_L>0\; :\; \forall x > 0,\; | \lambda^n \widetilde{\bar{l}}(\lambda, x)| \leq M_L.
\end{align*}
We conclude that $\forall x>0$, $\bar{l}(\cdot, x) \in C^\infty((0, \infty))$ and $\bar{l}(\cdot, x) \notin L_1((0, \infty))$. As $x=0$ in \eqref{lapLAMBDA}, we have the Laplace transform of the tail of a L\'{e}vy measure, that is the measure associated with $\bar{H}$.\\

We also notice that $\bar{H}_t$ exhibits non-decreasing paths, then 
\begin{align*}
\bar{L}_t = \inf\{s\geq 0\,:\, \bar{H}_s \notin (0,t)\}
\end{align*} 
can be also regarded as an exit time. On the other hand, the processes $\bar{L}_x$ and $\bar{H}_x$ can be respectively regarded as the lifetimes for $\bar{H}_t$, $t\geq 0$ and $\bar{L}_t$, $t\geq 0$ both with respect to the level $x \in [0, \infty)$.

The general theory on PDEs and subordinators with drift can be found in the book \cite[Section 8.4]{KoloBook11}. The fractional telegraph equation has been studied also in \cite{DovToaOrs} in a different setting and in \cite{FedeDov} with a general operator including the telegraph case. However, from the result in \cite{DovToaOrs} we are not able to obtain the equation \eqref{eq:RTL} which is therefore investigated here in more detail.   It can be easily proved that the solution to
\begin{align}
\label{telegrEQ}
\sigma \frac{\partial z}{\partial t} + D^\alpha_t z = \Delta z, \quad f \in D(\Delta)
\end{align}
has the probabilistic representation
\begin{align*}
z(t,x) = \mathbf{E}_x[f(B_{\bar{L}_t})]
\end{align*}
where $B$ is a Brownian motion with generator $(\Delta, D(\Delta))$. The equation \eqref{telegrEQ} is a fractional telegraph equation  only for $\alpha=1/2$. Here $\alpha \in (0,1]$, if we assume that $\alpha=2$, then \eqref{telegrEQ} becomes the telegraph equation for which the associated process is a telegraph process, say $T_t$, $t\geq 0$. The process $T_t$ would have the following representation
\begin{align*}
T_t = \int_0^t V_s\, ds \quad \textrm{where} \quad V_s =(-1)^{N_s}
\end{align*} 
and $N_s$, $s\geq 0$ is a Poisson process with parameter $\sigma> 0$.\\

The probabilistic representation of the solution to \eqref{telegrEQ} for $\alpha \in (1,2)$ is an open problem.


\section{Lifetimes for processes driven by a FBVP}
\label{SecMAIN}

Let $\widetilde{X} = (\widetilde{X}_t)_{t\geq 0}$ be the process with generator $(G, D(G))$ where $G=\Delta$ and
\begin{align*}
D(G) = \{ \varphi\,:\, \varphi, \Delta \varphi \in C((0, \infty))\,,\, \partial_x \varphi (0) = c \, \varphi(0) \}
\end{align*}
with elastic coefficient $c\geq 0$. The semigroup associated with the elastic Brownian motion $\widetilde{X}$ can be written by considering the couple $(X_t, M_t)_{t\geq 0}$ as discussed in the introduction, that is
\begin{align}
\label{probRepX}
\mathbf{E}_x[f(\widetilde{X}_t)] = \mathbf{E}_x[f(X_t)\, M_t], \quad f \in D(G).
\end{align}
Let us denote by
$$\gamma_t = \int_0^t \mathbf{1}_{(X_s=0)} ds$$
the local time at zero of the reflecting Brownian motion $X=(X_t)_{t\geq 0}$ on $[0, \infty)$. Then, the multiplicative functional $M_t = e^{A_t}$ is given by $A_t = - c\, \gamma_t$, $t\geq 0$. \\

We introduce the following random variables:
\begin{itemize}
\item[-] $\chi$ is of exponential type with parameter $c$, that is 
$$\mathbf{P}(\chi > x) = e^{-c x}, \quad x\geq 0;$$
\item[-] $\chi (\alpha)$ is of Mittag-Leffler type of order $\alpha \in (0,1)$ with parameter $c$, that is 
$$\mathbf{P}(\chi (\alpha) > x) = E_\alpha(-c x^\alpha),\quad x\geq 0.$$
\end{itemize}
As $\alpha \to 1$, 
\begin{align*}
\chi (\alpha) \to \chi (1) \stackrel{law}{=} \chi .
\end{align*}

Let us introduce the spaces
\begin{align*}
C^\alpha_L(I) = \left\lbrace u \in C(I)\,:\, \Big| \frac{d u}{d t}(t) \Big| \leq t^{\frac{\alpha}{2} - 1}, \; \forall\, t \in I \right\rbrace
\end{align*}
and
\begin{align*}
D^\alpha_L : = \left.
\begin{cases}
&  u : [0, \infty)\times [0, \infty) \mapsto (0, \infty)\; \textrm{such that:} \\
& a)\; C_b([0, \infty)) \ni u|_{t>0} : x \mapsto u(t,x), \\ 
& b)\; C^\alpha_L((0, \infty)) \ni u|_{x=0} : t \mapsto u(t,0)
\end{cases}
\right\rbrace
\end{align*}
where $C_b(I)$ as usual is the set of continuous and bounded functions on $I$.\\

We first consider the lifetime of the process associated with \eqref{PDEprocessMain}. 

\begin{theorem}
\label{thm:MAIN}
Let us consider
\begin{align*}
\alpha \in (0,1], \quad \eta \geq 0 \quad \sigma \geq 0, \quad c \geq 0.
\end{align*}
The solution $u \in C^{1,2}((0, \infty) \times [0, \infty); [0, \infty)) \cap D^\alpha_L$ to the problem 
\begin{equation}
\label{eqMAIN}
\begin{cases}
\displaystyle \frac{\partial u}{\partial t} (t,x) = \Delta u(t,x), \quad (t,x) \in (0, \infty) \times (0,\infty)\\ 
\displaystyle \eta D^{\alpha/2}_t u(t,0) = \sigma \frac{\partial u}{\partial x} (t, 0) -c\, u(t,0), \quad t>0\\
\displaystyle u(0,x) = \mathbf{1}_{[0, \infty)}(x), \quad x>0
\end{cases}
\end{equation}
has the probabilistic representation
\begin{align}
u(t,x) 
= & \mathbf{P}_x(t < \tau_0) + \mathbf{E}_x[e^{-\frac{c}{\eta} \bar{L}_{\gamma_t}}, t \geq \tau_0]
\end{align}
where $\tau_0= \inf\{t \geq 0\,:\, X_t =0\}$ and $\bar{L}_t = \inf\{s\geq 0\,:\, \bar{H}_s \geq t\}$ is an inverse to 
\begin{align*}
\bar{H}_t = \frac{\sigma}{\eta} t + H_t
\end{align*} 
which is independent from $X_t$.
\end{theorem}

The proof is postponed in Section \ref{SecProof}.\\

Let $Z$ be the process on $[0, \infty)$ driven by the problem \eqref{PDEprocessMain} with $\eta \geq 0$, $\sigma\geq 0$ and $c\geq 0$. This is the case of \eqref{eqMAIN}.

\begin{theorem}
\label{thm:lifeMAIN}
Let $\zeta$ be the lifetime of $Z$ started at $Z_0=x$, it holds that
\begin{equation*}
\zeta \stackrel{law}{=} \inf\{ s\geq 0\,:\, \gamma_s < \bar{H}_{\eta \chi} \}.
\end{equation*}
\end{theorem}
\begin{proof}
We observe that
\begin{align*}
\mathbf{P}_x(\zeta > t) = \mathbf{E}_x [e^{- \frac{c}{\eta} \bar{L}_{\gamma_t}} ] = \mathbf{P}_x(\eta \chi > \bar{L}_{\gamma_t}) = \mathbf{P}_x(\bar{H}_{\eta \chi} > \gamma_t),
\end{align*}
that is $\zeta \stackrel{law}{=} \inf\{ s\geq 0\,:\, \gamma_s < \bar{H}_{\eta \chi} \}$. 
\end{proof}

\begin{remark}
Notice that $\bar{H}_{\eta \chi} = \sigma \chi + H_{\eta \chi}$. As $\eta = 0$ we only have $\sigma \chi$ and the fractional derivative in \eqref{eqMAIN} disappears.
\end{remark}

Fractional boundary value problems do not seem to be related to the standard theory of the time changes. However, their probabilistic representations are still given in terms of random times to be additionally considered for a base process. Here the random times are non-decreasing processes related to equations of fractional telegraph-type. \\

Let us consider the elastic Brownian motion time-changed with $V_t$. The process $V_t$ can be considered as a time change of an independent Brownian motion in order to solve a fractional telegraph equation. If we consider $V_t$ as given in \eqref{RandomTimeV} where $\gamma_t$ is the local time of the elastic Brownian motion $(X_t, M_t)$, then the $\lambda$-potential is written as follows
\begin{align*}
R_\lambda^V f(x) 
= & \mathbf{E}_x\left[ \int_0^\infty e^{-\lambda t} f(X_{V^{-1}_t})\, \exp (- \frac{c}{\eta}\gamma_{V^{-1}_t} )\, dt \right]\\
= & \mathbf{E}_x\left[ \int_0^\infty e^{-\lambda V_t} f(X_{t})\,  e^{- \frac{c}{\eta} \gamma_t} \, dV_t \right] + R^{V,D}_\lambda f(x), \quad \lambda>0.
\end{align*}
After some manipulation, the $\lambda$-potential
\begin{align*}
\mathbf{E}_x\left[ \int_0^\infty e^{-\lambda V_t} f(X_{t})\,  e^{- \frac{c}{\eta} \gamma_t} \, dV_t \right]
\end{align*}
takes the form
\begin{align*}
\int_0^\infty e^{-\frac{\lambda}{\sigma} w - \frac{c}{\eta} w} \left( \int_0^\infty f(y) e^{-(x+y+w)\sqrt{\lambda}} dy + \frac{1}{\sigma} f(0)  e^{-(x+w)\sqrt{\lambda}} \right) dw
\end{align*}
and therefore
\begin{align*}
R_\lambda^V f(x)
\displaystyle  = & e^{-x \sqrt{\lambda}} \, \frac{ \sigma \int_0^\infty e^{-y \sqrt{\lambda}} f(y)dy + f(0)}{ c \sigma /\eta + \lambda + \sigma \sqrt{\lambda}} + R^{V,D}_\lambda f(x), \quad \lambda>0
\end{align*}
where
\begin{align*}
R^{V,D}_\lambda f(x) = \mathbf{E}_x \left[ \int_0^{\tau_0} e^{-\lambda t} f(X_{V^{-1}_t})\, dt  \right] = \mathbf{E}_x \left[ \int_0^{\tau_0} e^{-\lambda t} f(X_{t})\, dt  \right], \quad \lambda>0.
\end{align*}
In the present paper we consider the elastic coefficient $c/\eta$, the case $\eta=\sigma$ still represents a special setting. We get that
\begin{align*}
R_\lambda^V \mathbf{1}_{[0, \infty)}(x) 
\displaystyle  = & e^{-x \sqrt{\lambda}} \, \frac{ \sigma \lambda^{\frac{1}{2}-1} + 1}{ c \sigma /\eta + \lambda + \sigma \sqrt{\lambda}} + R^{V,D}_\lambda \mathbf{1}_{[0, \infty)}(x)\\
\displaystyle  = & e^{-x \sqrt{\lambda}} \frac{\lambda + \sigma \sqrt{\lambda}}{\lambda} \int_0^\infty e^{-\frac{c \sigma}{\eta} w} e^{-w (\lambda + \sigma \sqrt{\lambda})} dw + R^{V,D}_\lambda \mathbf{1}_{[0, \infty)}(x)
\end{align*}
which can be compared with \eqref{uZEROteleg}. Here
\begin{align*}
R^{V,D}_\lambda \mathbf{1}_{[0, \infty)}(x) =  \int_0^\infty e^{-\lambda t} \mathbf{P}_x(t < \tau_0)\, dt.
\end{align*}
As we can immediately see $R^V_\lambda$ is associated with \eqref{PDEprocessMain} as $\alpha \to 2$ which is not the case in our paper, indeed $\alpha \in (0, 1]$. We have that 
\begin{align*}
R_\lambda^V f(x) = \mathbf{E}_x[f(\widetilde{X}^V_t)]
\end{align*}
where $\widetilde{X}^V_t$ is identical in law to an elastic sticky Brownian motion.\\ 

The probabilistic representation of the solution to \eqref{PDEprocessMain} in case $\alpha \in (1,2)$ is an open problem. 


\section{An intuitive case}

An helpful intuitive reading is given by the fact that, from the heat equation 
$$D^1_t u = \Delta u,$$ 
we have that the boundary condition involving $D^1_t u(t,0)$ introduces a second-order boundary condition $\Delta u(t,0)$. On the other hand, the condition $D^{1/2}_t u(t,0)$ should introduce a condition involving the operator
\begin{align*}
-(-\Delta)^{1/2} u(t,0) = -\frac{\partial u}{\partial x} (t,0).
\end{align*}
The last identity may hold for functions extended with zero on the negative part of the real line, that is we are dealing with the Riesz operator $d \varphi/d|x|$ acting on $\varphi : [0, \infty) \to (0, \infty)$ such that $\varphi = 0$ on $(-\infty, 0)$. The definition of fractional Laplacian can be understood in the sense of Phillips (Bochner subordination). Thus, for $\beta \in (0,1]$,
\begin{align*}
(-\partial_x)^\beta \varphi(x) := \int_0^\infty \left( e^{-y\partial_x} \varphi(x) - \varphi(x) \right) \frac{y^{-\beta - 1}}{\Gamma(1-\beta)} dy
\end{align*}
is an operator of order $\beta$ where $e^{-y\partial_x} \varphi(x) =\varphi(x-y)$ is the translation semigroup and
\begin{align*}
(-\Delta)^\beta \varphi(x) := \int_0^\infty \left( e^{-y\Delta}\varphi(x) - \varphi(x) \right) \frac{y^{-\beta - 1}}{\Gamma(1-\beta)} dy
\end{align*}
is an operator of order $2\beta$. For $\varphi$ extended with zero on $(-\infty,0)$ the operator $(-\partial_x)^\beta$ corresponds to the Marchaud derivative of order $\beta \leq 1$ whereas, as $\beta=1/2$ the operator $(-\Delta)^\beta$ corresponds to a first order derivative. Thus, in the special case $\beta=1/2$, the equality
$$D^{1/2}_t u(t,0) = - \frac{\partial u}{\partial x}(t,0),$$
implies that the condition 
$$D^{1/2}_t u(t,0)= -c\, u(t,0)$$ 
plays the role of 
$$\frac{\partial u}{\partial x}(t,0) = c\,u(t,0)$$
which is indeed a Robin boundary condition.\\

As $\alpha \to 1$ in Theorem \ref{thm:MAIN} we obtain the following helpful and intuitive result.
\begin{corollary}
\label{coroHELPFUL} 
Let us consider
\begin{align*}
\eta \geq 0, \quad \sigma \geq 0, \quad c\geq 0.
\end{align*}
The solution $u \in C^{1,2}((0, \infty) \times [0, \infty); [0, \infty)) \cap D^{1}_L$ to the problem 
\begin{equation}
\label{eqALPHA1}
\begin{cases}
\displaystyle \frac{\partial u}{\partial t} (t,x) = \Delta u(t,x), \quad (t,x) \in (0, \infty) \times (0,\infty)\\ 
\displaystyle \eta D^{1/2}_t u(t,0) = \sigma \frac{\partial u}{\partial x}(t,0) - c\, u(t,0), \quad t>0\\
\displaystyle u(0,x) = \mathbf{1}_{[0, \infty)}(x), \quad x > 0
\end{cases}
\end{equation}
has the representation
\begin{align*}
u(t, x) = \mathbf{E}_x[\mathbf{1}_{[0, \infty)}(X_t) e^{- \frac{c}{\eta +\sigma} \gamma_t }]
\end{align*}
where $\gamma_t$ is the local time of $X_t$ started at $X_0=x \in [0, \infty)$. \end{corollary}
\begin{proof}
Since $\bar{L}_t = \frac{t}{1 + \sigma/\eta}$ almost surely as $\alpha=1$, we immediately get the result.
\end{proof}

The solution to \eqref{eqALPHA1} can be written in terms of the elastic Brownian motion with generator $(G^{\eta, \sigma},D(G^{\eta, \sigma}))$ where $G^{\eta, \sigma} = \Delta$ and 
\begin{align*}
D(G^{\eta, \sigma}) = \left\lbrace \varphi\,:\, \varphi, \Delta \varphi \in C((0, \infty)),\, \partial_x \varphi (0) = \frac{c}{\eta + \sigma} \varphi(0) \right\rbrace.
\end{align*}

We observe that, $L_{\gamma_t}$ equals in law $\gamma_{L_t}$. Thus, in order to have a clear picture about the role of the fractional boundary condition we have to consider $\sigma \neq 0$.

We also underline that for $\sigma<0$ the process $\bar{H}_t$ is a subordinator with negative drift. The special case $\eta+\sigma=0$ introduces the Dirichlet boundary condition (see the case $c\to \infty$ in Section \ref{sec:asymptotic}).


\section{The special case $\sigma=0$}
\label{SecSpecial}

First we recall that the lifetime $\zeta^{el}$ of an elastic Brownian motion on $[0, \infty)$ with Robin condition at zero can be written as
\begin{align*}
\zeta^{el} = \inf\{s \geq 0\,:\, \gamma^{el}_s < \chi \}
\end{align*}
where $\gamma^{el}$ is the corresponding local time at zero and $\chi$ is an exponential random variable. We show that the fractional boundary value problem introduces the Mittag-Leffler random variable $\chi(\alpha)$ in place of the exponential random variable $\chi$. Moreover, as $\alpha \to 1$, the lifetime $\zeta$ equals in law the lifetime $\zeta^{el}$. Without loss of generality we assume that $\eta=1$.

\begin{corollary}
\label{coro:special}
Let us consider
\begin{align*}
\alpha \in (0,1], \quad c\geq 0.
\end{align*}
The solution $u \in C^{1,2}((0, \infty) \times [0, \infty); [0, \infty)) \cap D^\alpha_L$ to the problem 
\begin{equation}
\label{eq:special}
\begin{cases}
\displaystyle \frac{\partial u}{\partial t} (t,x) = \Delta u(t,x), \quad (t,x) \in (0, \infty) \times (0,\infty)\\ 
\displaystyle D^{\alpha/2}_t u(t,0) = - c\, u(t,0), \quad t>0\\
\displaystyle u(0,x) = \mathbf{1}_{[0, \infty)}(x), \quad x > 0
\end{cases}
\end{equation}
has the representation
\begin{align}
\label{sol:special}
u(t,x) 
= & \mathbf{P}_x(t < \tau_0) + \mathbf{E}_x[e^{- c L_{\gamma_t}}, t \geq \tau_0]
\end{align}
where $\tau_0=\inf\{t \geq 0\,:\, X_t =0\}$, $\gamma_t$ is the local time at zero of $X_t$ started at $X_0=x$ and $L_t$ is an inverse to an $\alpha$-stable subordinator independent from $X_t$.
\end{corollary}

Let $Y$ be the process on $[0, \infty)$ driven by \eqref{PDEprocessMain} with $\sigma=0$ and $c\geq 0$. This is the case of \eqref{eq:special}.

\begin{corollary}
\label{coro:lifeSpecial}
For the lifetime $\zeta$ of the process $Y$ started at $Y_0=x$, it holds that
\begin{align}
\label{ZETAlifeSpecial}
\zeta \stackrel{law}{=} \inf\{ s\geq 0 \,:\, 0 \leq \gamma_s < \chi (\alpha) \}. 
\end{align}
\end{corollary}

\begin{proof}
From Corollary \ref{coro:special}, we deduce that 
\begin{align*}
\mathbf{P}_x(\zeta > t) = \mathbf{E}_x[e^{-c L_{\gamma_t}}] = \mathbf{P}_x(\chi > L_{\gamma_t}) = \mathbf{P}_x(H_\chi > \gamma_t).
\end{align*}
Since
\begin{align*}
\mathbf{P}(H_\chi < s) 
= & \int_0^\infty c \,e^{-ct} \, \mathbf{P}(H_t < s)\, dt\\ 
= & 1 - c \int_0^\infty e^{-c t} \,\mathbf{P}(L_s < t)\, dt\\ 
= & 1 - E_\alpha(-c s^\alpha)
\end{align*}
the variable $H_\chi$ follows a Mittag-Leffler distribution. In particular $H_\chi$ equals in law $\chi (\alpha)$ and 
\begin{align*}
\mathbf{P}_x(\gamma_t < H_\chi) = \mathbf{P}_x(\gamma_t < \chi (\alpha))
\end{align*}
from which the result follows.
\end{proof}

As a by-product of Corollary \ref{coro:lifeSpecial} we underline the following facts:
\begin{itemize}
\item[i)] If $Y_0 = x \in [0, \infty)$, then
\begin{align}
\label{zetaLifeSpecial}
\zeta \stackrel{law}{=} \tau_0 + \chi(\alpha/2).
\end{align}
Since 
\begin{align*}
\mathbf{P}_x(\zeta >t) = \mathbf{P}_x(\tau_0 >t) + \mathbf{P}_x(0 < \gamma_t < \chi(\alpha))
\end{align*}
we get that
\begin{align*}
\int_0^\infty e^{-\lambda t} \mathbf{P}_x(\zeta >t) \, dt 
= & \frac{1 - e^{-x\sqrt{\lambda}}}{\lambda} + \frac{\lambda^{\frac{\alpha}{2}-1}}{c+\lambda^{\frac{\alpha}{2}}} e^{-x \sqrt{\lambda}}\\
= & \frac{1}{\lambda} - \frac{1}{\lambda} \frac{c}{c + \lambda^\frac{\alpha}{2}} e^{-x \sqrt{\lambda}}.
\end{align*}
Now we observe that
\begin{align*}
\int_0^\infty e^{-\lambda t} \mathbf{P}_x(\tau_0 + \chi(\alpha/2) >t)\, dt = \frac{1}{\lambda} - \frac{1}{\lambda} \mathbf{E}_x[e^{-\lambda (\tau_0 + \chi(\alpha/2))}]
\end{align*}
where $\tau_0$ is independent from $\chi(\alpha/2)$. Moreover, 
\begin{align*}
\mathbf{E}_x[e^{-\lambda \tau_0}] = e^{-x \sqrt{\lambda}}.
\end{align*}
Since
\begin{align*}
\mathbf{E}[e^{-\lambda \chi(\alpha/2)}]
= & \int_0^\infty e^{-\lambda t} \frac{d}{dt}(1- E_\frac{\alpha}{2}(-ct^\frac{\alpha}{2})) dt\\
= & \lambda \left( \frac{1}{\lambda} - \frac{\lambda^{\frac{\alpha}{2}-1}}{c + \lambda^\frac{\alpha}{2}} \right)\\
= & \frac{c}{c + \lambda^\frac{\alpha}{2}}
\end{align*}
we get that
\begin{align}
\label{calcLIFEspecial}
\int_0^\infty e^{-\lambda t} \mathbf{P}_x(\tau_0 + \chi(\alpha/2) >t)\, dt = \frac{1}{\lambda} - \frac{1}{\lambda} \frac{c}{c+\lambda^\frac{\alpha}{2}} e^{-x \sqrt{\lambda}}.
\end{align}

\item[-] If $Y_0=0$, then
\begin{align*}
\mathbf{P}_0(\zeta > t) = E_\frac{\alpha}{2}(-c t^\frac{\alpha}{2}), \quad t\geq 0,
\end{align*} 
that is 
\begin{align*}
\zeta \stackrel{law}{=} \chi(\alpha /2).
\end{align*}
Indeed, fix $x=0$ and denote by $\gamma^{-1}$ the inverse of $\gamma$. Since $\gamma_t$ equals in law an inverse to a $1/2$-stable subordinator, the composition $\gamma^{-1}_{H_t}$ equals in law an $\alpha/2$-stable subordinator. Thus, $(\gamma_t < \chi (\alpha)) \equiv (t < \gamma^{-1}_{H_\chi})$ under $\mathbf{P}_0$. We conclude that,
for the process $Y$ started at $Y_0=0$, the lifetime $\zeta$ is identical in law to a Mittag-Leffler random variable of order $\alpha/2 \in (0, 1/2)$ with parameter $c\geq 0$. 
\end{itemize}

For a time-changed Markov process (that is for the fractional Cauchy problem \eqref{FCP} and the Cauchy problem involving a general generator) we have shown in \cite{DelRus} that its mean lifetime is given by $\mathbf{E}_x[H_\zeta]$ (for the composition with $L$) and by  $\mathbf{E}_x[L_\zeta]$ (for the composition with $H$) where $\zeta$ is the lifetime of the base process. The so-called delayed process (the process time-changed by $L$) has infinite mean lifetime and we say that the process spends an infinite mean amount of time in a domain. This is the case of the problem \eqref{FCP}, due to the fact that $\mathbf{E}_0[H_t]$ is infinite, we still have an infinite mean lifetime. The base process is delayed along its path. Here we obtain that the lifetime \eqref{zetaLifeSpecial} has a distribution related with a Mittag-Leffler random variable. In particular, for the process started at $x \in [0, \infty)$ the mean lifetime $\mathbf{E}_x[\zeta]$ is written in terms of $\mathbf{E}_x[\tau_0]$ and $\mathbf{E}_0[\chi(\alpha/2)]$. The corresponding elliptic problem does not exist for $c\neq 0$. Indeed, the mean value of $\chi(\alpha/2)$ turns out to be finite only if $c=0$ (recall that $\chi(\alpha/2)$ is a Mittag-Leffler random variable with parameters $c\geq 0$ and $\alpha/2 \in (0, 1/2]$). We have a finite mean lifetime for the process driven by $\partial_t u = - (-\Delta)^\alpha u$ for which we have a mean lifetime $\mathbf{E}_x[L_\zeta]$ with $\mathbf{E}_0[L_t] = C(\alpha)\, t^{\alpha}$ (see \cite{DelRus}).


\section{Some asymptotic results}
\label{sec:asymptotic}

We discuss the following results moving from the fact introduced by Feller (1958) as a conjecture (Feller conjecture) which has been well-studied after its statement. The boundary condition 
\begin{align*}
(1-\rho)\, u^\prime(0) = \rho\, u(0), \quad \rho \in [0, 1]
\end{align*}  
can be considered in order to study the Robin boundary condition and therefore, the limit cases corresponding to the Dirichlet and the Neumann boundary conditions. By restating the previous condition as
\begin{align*}
u^\prime(0) = \varrho \, u(0), \quad \varrho = \frac{\rho}{1-\rho}
\end{align*}
we formally get the Dirichlet condition as $\varrho \to \infty$ and the Neumann condition as $\varrho \to 0$. We now use similar arguments in case of the fractional boundary condition
\begin{align*}
\eta D^{\alpha/2}_t u(t,0) = \sigma \frac{\partial u}{\partial x}(t,0) - c\, u(t,0).
\end{align*}
We have the following formal cases (for the sake of simplicity we set $\eta=1$):
\begin{itemize}
\item[-] $\forall\, c > 0$, as $\sigma \to 0$ we get the result in the previous Corollary \ref{coro:special};
\item[-] $\forall\, c > 0$, as $\sigma \to \infty$ we formally get the Neumann condition $\partial_x u(t,0)= 0$. The solution takes the form
\begin{align*}
u(t,x) = Q^D_t \mathbf{1}_{[0, \infty)} + \int_0^t \frac{x}{\tau} g(\tau, x)\, d\tau. 
\end{align*}
Since 
\begin{align}
Q^D_t \mathbf{1}_{[0, \infty)} = 1 - \int_0^t \frac{x}{\tau} g(\tau, x)\, d\tau
\end{align}
we obtain that 
\begin{align*}
1=u(t,x)= \mathbf{P}_x(\zeta >t), \; t\geq 0, \; x \in [0, \infty).
\end{align*}
The process has infinite lifetime;
\item[-] $\forall \sigma \geq 0$ as $c\to 0$ we have that
\begin{align*}
u(t,x) = \mathbf{E}_x[\mathbf{1}_{[0, \infty)}(X_t)] = 1
\end{align*}
that is $\forall x$
\begin{align*}
\mathbf{P}_x(\zeta >t) = 1, \quad \forall t
\end{align*}
The lifetime is infinite. The process behaves like a Brownian motion reflected at $x=0$;
\item[-] $\forall\, \sigma \geq 0$, as $c\to \infty$ we get 
\begin{align*}
u(t,x) = \mathbf{E}_x[\mathbf{1}_{[0, \infty)}(X_t), \gamma_t=0] = \mathbf{E}_x[\mathbf{1}_{[0, \infty)}(X_t), t < \tau_0]
\end{align*}
where $\tau_0$ is the first time the process $X_t$ hits the point zero. Thus the reflecting Brownian motion $X$ is stopped at the random time $\tau_0$. We get $u(t,x) = \mathbf{P}_x(\tau_0 > t)$ and the corresponding process behaves like a Brownian motion killed on the boundary point $x=0$;
\item[-] Suppose that $c/\sigma \to \rho $ as $c,\sigma \to \infty$, then 
\begin{align*}
u(t,x) = \mathbf{E}_x[\mathbf{1}_{[0, \infty)}(X_t) \, e^{-\rho \gamma_t}]
\end{align*}
and
\begin{align*}
\mathbf{P}_x(\zeta >t) = \mathbf{P}_x(\gamma_t < \chi)
\end{align*}
where $\chi \sim Exp(\rho)$. The corresponding process behaves like an elastic Brownian motion.
\end{itemize}
Moreover, we may consider the following cases (for $\sigma, c \in (0,\infty)$):
\begin{itemize}
\item[-] as $\eta\to 0$, we have the Robin boundary condition;  
\item[-] as $\eta \to \infty$, we have the Neumann boundary condition  corresponding to
\begin{align*}
D^{\alpha/2}_t u(t,0)=0, \quad t>0.
\end{align*}
\end{itemize}


\section{Proof of the results}

\label{SecProof}

Let us consider 
\begin{align}
\label{SemigQ}
Q_t f(x) = \mathbf{E}_x[f(\widetilde{X}_t)]
\end{align}
where $\widetilde{X}_t$ with generator $(G, D(G))$ has been introduced in Section \ref{SecMAIN}. We obtain a characterization of \eqref{SemigQ} by exploiting the fact that $Q_t f(x) = u(t,x)$ can be written as
\begin{align}
\label{densityu}
u(t,x) = Q^D_t f(x) + \int_0^t \frac{x}{\tau} g(\tau,x)\, u(t-\tau, 0)\, d\tau 
\end{align}
where
\begin{align}
\label{semigQD}
Q^D_t f(x) = \int_0^\infty \big( g(t,x-y) - g(t, x+y) \big)\, f(y)\, dy
\end{align}
and $g(t,z)= e^{-z^2/4t} / \sqrt{4\pi t}$ is the Gaussian kernel. The semigroup \eqref{SemigQ} can written as
\begin{align*}
Q_t f(x) = \mathbf{E}_x[f(X_t) M_t]
\end{align*}
where $X_t$ is a Brownian motion reflected at $x=0$ and $M_t$ is the multiplicative functional associated with the boundary condition. Thus, $\widetilde{X}_t$ is the part process of $X_t$ on $[0, \infty)$ with lifetime given by $M_t$. The process $M_t$ is related to the additive functional $A_t = -\ln M_t$. In case of the elastic Brownian motion for instance, $A_t = c\, \gamma_t$ where $\gamma_t$ is the local time at $\{0\}$ for $X_t$. Let us write
\begin{align}
\label{repQ}
& \int_\mathbb{R} f(y) \mathbf{P}_x(X_t \in dy, \gamma_t \in dw)\\ 
= & \int_0^\infty f(y)\, \frac{x+y+w}{t} \, g(t, x+y+w)\,  dy\, dw + Q^D_t f(x)\, \delta(w)dw\notag . 
\end{align}
and
\begin{align}
\label{repQbar}
& \int_\mathbb{R} f(y) \mathbf{P}_x(X_t \in dy, \bar{L}_{\gamma_t} \in dw) \\
= & \int_0^\infty f(y)\left( \int_0^\infty \frac{x+y+z}{t} \, g(t, x+y+z)\, \bar{l}(z,w)\, dz \right) dy\, dw\notag \\
& + Q^D_t f(x)\, \delta(w)dw\notag . 
\end{align}
From \eqref{repQ}, we have that
\begin{align*}
Q_t f(x) = \int_0^\infty \int_0^\infty f(y) \, e^{-c w}\, \mathbf{P}_x(X_t \in dy, \gamma_t \in dw) + Q_t^D f(x)
\end{align*}
for which $Q_t Q_s f(x) = Q_{t+s}f(x)$. Indeed, the elastic Brownian motion is a (strong) Markov process for which $Q_{t+s}f$ is gievn by
\begin{align}
\mathbf{E}_x[f(X_{t+s}) g(\gamma_{t+s})]
= & \mathbf{E}_x[\mathbf{E}_0[f(X_{t+s} - X_s + X_s) g(\gamma_{t+s} - \gamma_s + \gamma_s) | X_s]]\notag \\ 
= & \mathbf{E}_x[\mathbf{E}_0 [ f(X_t + X_s) g(\gamma_t + \gamma_s) | X_s]] \notag \\
= & \mathbf{E}_x [\mathbf{E}_{X_s}[f(X_t) g(\gamma_t)]]
\end{align}
with $g(s)=e^{-cs}$. Let us introduce 
\begin{align}
\label{relationK}
K_\alpha( \kappa s ) = \mathbf{E}_0[e^{- \kappa \bar{L}_{s}}], \quad s \geq 0, \quad \kappa \geq 0.
\end{align}
From \eqref{repQbar}, we can analogously write $\bar{Q}_t f(x) := \mathbf{E}_x[f(X_t) \bar{M}_t]$, that is
\begin{align*}
\bar{Q}_t f(x) 
= & \int_0^\infty \int_0^\infty f(y) \, e^{-\frac{c}{\eta} w}\, \mathbf{P}_x(X_t \in dy, \bar{L}_{\gamma_t} \in dw) + Q_t^D f(x)\\
= & \int_0^\infty \int_0^\infty f(y) \, K_\alpha( \frac{c}{\eta} z)\, \mathbf{P}_x(X_t \in dy, \gamma_t \in dz) + Q_t^D f(x).
\end{align*}
As $\alpha = 1$ formula \eqref{relationK} takes the form $K_1(\kappa s) = e^{-\kappa s}$ and $Q_t=\bar{Q}_t$. Indeed, $\bar{L}_t \to t$ almost surely as $\alpha \to 1$. Since $|K_\alpha| \leq 1$ for any $\alpha \in (0,1]$ we therefore have that
\begin{align*}
|\mathbf{E}_x[f(X_t) K_\alpha(\frac{c}{\eta} \gamma_t)]| \leq \mathbf{E}_x[| f(X_t)|] \leq \|f\|_{\infty} \quad \forall\, x \in [0, \infty)
\end{align*}
from which we deduce that $\|\bar{Q}_t f\|_{\infty} \leq \|f\|_\infty$ for any $f \in C_b([0, \infty))$ where $\|f\|_\infty = \sup_x |f(x)|$. Thus $|\bar{Q}_t | \leq 1$ and $\bar{Q}_t$ is bounded as expected.\\

In the special case $\sigma=0$ of Section \ref{SecSpecial}, the previous arguments lead to $K_\beta(cs) = E_\beta(-c s^\beta)$, $s \geq0$, $\beta=\alpha/2$, $c\geq 0$ for which
\begin{align*}
K_\alpha( \kappa s ) = \mathbf{E}_0 [K_1(\kappa\, L_s)], \quad s \geq 0, \quad \kappa \geq 0.
\end{align*}

For the sake of completeness we provide the proof of the following result which will be useful further on. The result has been stated in \cite{BerBook} without the proof.
\begin{lemma}
For the densities $l$ and $h$ of the processes $L$ and $H$, it holds that 
\label{BertoinNOproof}
\begin{align*}
\alpha \frac{x}{t} l(t, x) = h(x, t), \quad t >0, \; x >0.
\end{align*}
\end{lemma}
\begin{proof}
\begin{align*}
\int_0^\infty \int_0^\infty e^{-\lambda t- \xi x} \alpha\, l(t,x)\, dx\, dt 
= & \int_0^\infty e^{-\lambda t} \alpha\, E_\alpha(- \xi t^\alpha)\, dt\\
= & \alpha \frac{\lambda^{\alpha - 1}}{\xi + \lambda^\alpha}\\
= & \int_0^\infty e^{-\xi x} \alpha\, \lambda^{\alpha-1} e^{-x \lambda^\alpha}\, dx\\
= & - \frac{d}{d\lambda} \int_0^\infty e^{-\xi x} \frac{1}{x} e^{-x \lambda^\alpha}\, dx\\
= & - \frac{d}{d\lambda} \int_0^\infty \int_0^\infty e^{-\xi x -\lambda t} \frac{1}{x} h(x, t)\, dt\, dx\\
= & \int_0^\infty \int_0^\infty e^{-\xi x -\lambda t} \frac{t}{x} h(x, t)\, dt\, dx
\end{align*}
\end{proof}

Moreover, we are in need of the following facts.

\begin{lemma}
\label{lemmaConvolution}
We have that
\begin{align}
\label{lmConv1}
\int_0^t \frac{x}{t- \tau} g(t-\tau, x)\, g(\tau, y)\, d\tau = g(t, x+y)
\end{align}
and
\begin{align}
\label{lmConv2}
\int_0^t \frac{x}{t- \tau} g(t-\tau, x)\, \frac{y}{\tau}
 g(\tau, y)\, d\tau = \frac{x+y}{t} g(t, x+y).
\end{align}
\end{lemma}
\begin{proof}
We recall that $2g(t,x)$ can be regarded as the density of an inverse to $1/2$-stable subordinator. Thus,
\begin{align*}
\int_0^\infty e^{-\lambda t} g(t,x)\, dt = \frac{1}{2} \lambda^{\frac{1}{2}-1} e^{-x \sqrt{\lambda}}
\end{align*}
and, in view of Lemma \ref{BertoinNOproof}, 
\begin{align*}
\int_0^\infty e^{-\lambda t} \frac{x}{t}g(t,x)\, dt = e^{-x \sqrt{\lambda}}.
\end{align*}
The convolution \eqref{lmConv1} leads to the Laplace transform
\begin{align*}
\frac{1}{2} \lambda^{\frac{1}{2}-1} e^{-(x+y)\sqrt{\lambda}} = \int_0^\infty e^{-\lambda t} g(t,x+y)\, dt
\end{align*}
where the identity is justified once again by Lemma \ref{BertoinNOproof}.

Formula \eqref{lmConv2} can be obtained by considering the same arguments and the fact that
\begin{align*}
e^{-(x+y)\sqrt{\lambda}} = \int_0^\infty e^{-\lambda t} \frac{x+y}{t} g(t,x+y)\, dt.
\end{align*}
\end{proof}

We move to the proof of the main result of the work.

\begin{proof}[Proof of Theorem \ref{thm:MAIN}]
Let us assume that the solution to \eqref{eqMAIN} has the probabilistic representation $u(t,x) = \mathbf{E}_x[f(X_t)\bar{M}_t]$. Observe that 
\begin{align*}
\mathbf{E}_x[f(X_t)\bar{M}_t] = \mathbf{E}_x[f(X_t)\bar{M}_t, \gamma_t=0] + \mathbf{E}_x[f(X_t)\bar{M}_t, \gamma_t>0]
\end{align*}
which can be written in terms of the first time the process $X$ hits the point $x=0$, that is
\begin{align*}
\tau_0 = \inf \{t \geq 0\,:\, X_t = 0\}.
\end{align*}
Let us introduce the $\lambda$-potential
\begin{align*}
R_\lambda f(x) = \mathbf{E}_x \left[ \int_0^\infty e^{-\lambda t} f(X_t)\bar{M}_t\, dt \right], \quad \lambda >0
\end{align*}
for which we write
\begin{align*}
R_\lambda f(x) = R^D_\lambda f(x) + \bar{R}_\lambda f(x)
\end{align*}
where
\begin{align*}
R^D_\lambda f(x) = \mathbf{E}_x\left[ \int_0^{\tau_0} e^{-\lambda t} f(X_t) \bar{M}_t\, dt \right]
\end{align*}
and
\begin{align*}
\bar{R}_\lambda f(x) = \mathbf{E}_x \left[ \int_{\tau_0}^\infty e^{-\lambda t} f(X_t) \bar{M}_t\, dt \right].
\end{align*}
We have that
\begin{align*}
\Delta R_\lambda f(x) 
= & \lambda R_\lambda f(x) - f(x)\\
= & \left( \lambda R^D_\lambda f(x) - f(x) \right) + \lambda \bar{R}_\lambda f(x)
\end{align*}
with
\begin{align*}
\Delta R^D_\lambda f(x) = \lambda R^D_\lambda f(x) - f(x) \quad \textrm{and} \quad \Delta \bar{R}_\lambda f(x) = \lambda \bar{R}_\lambda f(x).
\end{align*}
On the other hand
\begin{align*}
\int_0^\infty e^{-\lambda t} D^{\alpha/2}_t u(t,x)\, dt 
= & \lambda^{\alpha/2} R_\lambda f(x) - \lambda^{\alpha/2-1} f(x) \\
= & \lambda^{\alpha/2} R^D_\lambda f(x) + \lambda^{\alpha/2-1} \left( \lambda \bar{R}_\lambda f(x) - f(x) \right)
\end{align*}
with
\begin{align*}
R^D_\lambda f(x) \big|_{x=0} = 0.
\end{align*}
Thus, the potential $\bar{R}_\lambda$ completely characterizes the boundary value problem in terms of the fractional derivative. 

The Laplace transform of \eqref{densityu} is given by
\begin{align*}
\widetilde{u} (\lambda, x) 
= & \int_0^\infty e^{-\lambda t} u(t,x)\, dt\\
= & \int_0^\infty e^{-\lambda t} Q^D_t f(x)\, dt + \int_0^\infty e^{-\lambda t} \int_0^t \frac{x}{\tau} g(\tau, x) \, u(t-\tau, 0)\, d\tau\, dt
\end{align*}
where 
\begin{align*}
\int_0^\infty e^{-\lambda t} \int_0^t \frac{x}{\tau} g(\tau, x) \, u(t-\tau, 0)\, d\tau\, dt = e^{-x \sqrt{\lambda}}\, \widetilde{u}(\lambda, 0) = \bar{R}_\lambda f(x).
\end{align*}
We obviously have that
\begin{align*}
\lambda R^D_\lambda f(x) \to f(x) \quad \textrm{as} \quad \lambda \to \infty
\end{align*}
whereas, as $\lambda \to \infty$
\begin{align*}
\lambda \bar{R}_\lambda f(x) = e^{-x \sqrt{\lambda}} \lambda R_\lambda f(0)  \to 
\left\lbrace
\begin{array}{ll}
\displaystyle f(0), & x=0\\
\displaystyle 0, & x>0
\end{array}
\right .
\end{align*}
assuming that $\lambda R_\lambda \to f$ as $\lambda \to \infty$. Our assumption is therefore on the continuity of the paths of the involved process, that is
\begin{align*}
\lim_{t\to 0^+} \mathbf{E}_x[f(X_t) e^{-c \bar{L}_{\gamma_t}}] = f(x)
\end{align*}
which is verified (for the analytic check see the formula \eqref{uZEROpotential} below).

Now we observe that
\begin{align}
\label{usefulBC}
&\int_0^\infty e^{-|x-y|} f(y)\, dy \notag \\
= & \int_0^x e^{-(x-y)} f(y)\, dy + \int_x^\infty e^{-(y-x)} f(y)\, dy\notag \\
= & \int_0^x e^{-(x-y)} f(y)\, dy + \int_0^\infty e^{-(y-x)}f(y)\, dy - \int_0^x e^{-(y-x)} f(y)\, dy.
\end{align}
With \eqref{usefulBC} in mind, we write
\begin{align*}
\int_0^\infty e^{-\lambda t} Q^D_t f(x)\, dt
= & \frac{1}{2} \int_0^\infty \left( \frac{e^{-|x-y|\sqrt{\lambda}}}{\sqrt{\lambda}} - \frac{e^{-(x+y)\sqrt{\lambda}}}{\sqrt{\lambda}} \right) f(y)\, dy\\
= & \frac{1}{2} \int_0^\infty \frac{e^{(x-y)\sqrt{\lambda}}}{\sqrt{\lambda}} f(y)\, dy - \frac{1}{2} \int_0^\infty \frac{e^{-(x+y)\sqrt{\lambda}}}{\sqrt{\lambda}} f(y)\, dy \\
& - \frac{1}{2} \int_0^x \left( \frac{e^{(x-y)\sqrt{\lambda}}}{\sqrt{\lambda}} - \frac{e^{-(x-y)\sqrt{\lambda}}}{\sqrt{\lambda}} \right) f(y)\, dy.
\end{align*}
From this standard calculation we obtain $\widetilde{u}(\lambda, x)$ written in a convenient form for our purposes. A quick check shows that
\begin{align*}
\Delta \widetilde{u}(\lambda, x) = \lambda \, \widetilde{u}(\lambda, x) - f(x) = \int_0^\infty e^{-\lambda t} \frac{\partial u}{\partial t} (t,x)\, dt.
\end{align*}

The derivative \eqref{def:caputoDer} is defined as a convolution operator. Thus, the Laplace transform is given by
\begin{align*}
\int_0^\infty e^{-\lambda t} D^{\alpha/2}_t u(t,x)\, dt 
= & \lambda^{\alpha/2 -1} \left(\lambda \widetilde{u}(\lambda,x) - u(0,x) \right) \\
= & \lambda^{\alpha/2} \widetilde{u}(\lambda, x) - \lambda^{\alpha/2 -1} f(x), \quad \lambda>0.
\end{align*}
By taking into account the boundary condition, we arrive at the identity 
\begin{align}
\eta \lambda^{\alpha/2} \widetilde{u}(\lambda, x) - \eta \lambda^{\alpha/2 -1}  f(x) \bigg|_{x=0} = \sigma \frac{\partial \widetilde{u}}{\partial x}(\lambda, x) - c \widetilde{u}(\lambda, x) \bigg|_{x=0}
\end{align}
from which we obtain 
\begin{align}
\label{uZEROpotential}
\widetilde{u}(\lambda, 0) = \frac{\sigma \int_0^\infty e^{-y \sqrt{\lambda}} f(y) dy + \eta \lambda^{\alpha/ 2-1} f(0)}{c + \sigma \sqrt{\lambda} + \eta \lambda^{\alpha/2}},
\end{align}
that is,
\begin{align*}
\bar{R}_\lambda f(x) \big|_{x=0}.
\end{align*}
Since $f(x)=\mathbf{1}_{[0, \infty)}$ we write
\begin{align}
\widetilde{u}(\lambda,0) 
= & \left( \sigma \lambda^{1/2 - 1}  + \eta \lambda^{\alpha/ 2-1}\right) \int_0^\infty e^{-c w} e^{-\sigma w \sqrt{\lambda}} e^{-w \eta  \lambda^{\alpha/2}} dw \notag \\
= & \frac{\sigma \lambda^{1/2} + \eta \lambda^{\alpha/2}}{\lambda} \int_0^\infty e^{-c w} e^{-\sigma w \sqrt{\lambda}} e^{-w \eta  \lambda^{\alpha/2}} dw. \label{uZEROteleg}
\end{align}
Actually, this is the real advantage in studying $f=\mathbf{1}$. We observe that, for a general datum $f \in C_b([0, \infty))$, the solution can not be written as $\bar{Q}_t f$. Let us write
\begin{align*}
\widetilde{U}_1 (\lambda) = \sigma \int_0^\infty  e^{-c w}  \, \lambda^{1/2-1}  e^{-w \eta \lambda^{\alpha/2}} e^{-\sigma w \sqrt{\lambda}} \, dw\, dy
\end{align*}
and
\begin{align*}
\widetilde{U}_2(\lambda) = \eta \int_0^\infty e^{-c w} \, \lambda^{\alpha/ 2-1} e^{- w \eta \lambda^{\alpha/2}} e^{-\sigma w \sqrt{\lambda}}\, dw
\end{align*}
We use the fact that
\begin{align}
& \displaystyle \int_0^\infty e^{-\lambda t} \frac{w}{t}g(t,w)\, dt = e^{- w\sqrt{\lambda}}\label{f0U}\\
& \displaystyle \int_0^\infty e^{-\lambda t} g(t,w)\, dt = \lambda^{1/2-1} e^{- w\sqrt{\lambda}}\label{f1U}\\
& \displaystyle \int_0^\infty e^{-\lambda t} \int_0^\infty \frac{s}{t} g(t, s)\, \alpha \frac{w}{s} l(s,w)\, ds\, dt = e^{-w \lambda^{\alpha/2}} \label{f2U}\\
& \displaystyle \int_0^\infty e^{-\lambda t} \int_0^\infty g(t,s)\, l(s,w)\, ds\, dt = \lambda^{\alpha/2 -1} e^{-w \lambda^{\alpha/2}} \label{f3U}
\end{align}
and we obtain the inverse Laplace transforms $U_1(t)$, $U_2(t)$ by means of which we write
\begin{align}
\label{AllAtOnce}
\int_0^t \frac{x}{t-\tau} g(t-\tau, x)\, u(\tau,0)\, d\tau = \int_0^t \frac{x}{t-\tau} g(t-\tau, x)\,  \left( U_1(\tau) + U_2(\tau) \right) \, d\tau.
\end{align}
In particular, for \eqref{AllAtOnce} we show that
\begin{align*}
\int_0^\infty e^{-\lambda t} \left( \int_0^t \frac{x}{t-\tau} g(t-\tau, x)\,  \left( U_1(\tau) + U_2(\tau) \right) \, d\tau \right)dt = \bar{R}_\lambda f(x).
\end{align*}

\noindent {\it i) The Laplace transform $\widetilde{U}_1(\lambda)$}.\\

From \eqref{f1U} and \eqref{f2U},
\begin{align*}
U_1(t) = \sigma \int_0^\infty e^{-c w} \int_0^t g(t-\tau, \sigma w) \int_0^\infty \frac{s}{\tau} g(\tau, s)\, \alpha \frac{\eta w}{s} l(s, \eta w)\, ds\, d\tau\, dw
\end{align*}
By considering Lemma \ref{lemmaConvolution}, we get
\begin{align*}
\int_0^t g(t-\tau, \sigma w) \frac{s}{\tau} g(\tau, s)\, d\tau = g(t, s+\sigma w)
\end{align*}
by means of which we write the inverse of $\widetilde{U}_1(\lambda)$, that is
\begin{align*}
U_1(t) = \sigma \int_0^\infty  e^{-c w} \int_0^\infty g(t, s+\sigma w)\, \alpha \frac{\eta w}{s} l(s,\eta w)\, ds\, dw
\end{align*}
By applying once again Lemma \ref{lemmaConvolution},
\begin{align}
& \int_0^t \frac{x}{t-\tau} g(t-\tau, x) U_1(\tau)\, d\tau \label{onceAgain1}\\
= & \sigma \int_0^\infty e^{-c w} \int_0^\infty g(t, x+s+\sigma w)\, \alpha \frac{\eta w}{s} l(s,\eta w)\, ds\, dw \notag\\
= & \sigma \int_0^\infty e^{-c w} \int_{\sigma w}^\infty g(t, x+z)\, \alpha \frac{\eta w}{z-\sigma w} l(z-\sigma w,\eta w)\, dz\, dw. \notag
\end{align}

\noindent {\it ii) The Laplace transform $\widetilde{U}_2(\lambda)$}.
\\

From \eqref{f0U} and \eqref{f3U},
\begin{align*}
U_2(t) 
= & \eta \int_0^\infty e^{-cw} \int_0^t \frac{\sigma w}{t-\tau} g(t-\tau, \sigma w) \int_0^\infty g(\tau, z)\, l(z,\eta w)\, dz\, d\tau\\
= & \eta \int_0^\infty \int_0^\infty e^{-cw} \left( \int_0^t \frac{\sigma w}{t-\tau} g(t-\tau, \sigma w) \, g(\tau, z)\, d\tau \right)  l(z,\eta w)\, dz
\end{align*}
By considering Lemma \ref{lemmaConvolution}, we get
\begin{align*}
\int_0^t \frac{\sigma w}{t-\tau} g(t-\tau, \sigma w) \,  g(\tau, z)\, d\tau = g(t, z + \sigma w)
\end{align*}
and
\begin{align*}
U_2(t) = \eta \int_0^\infty \int_0^\infty e^{-cw}\, g(t, z+\sigma w) \, l(z,\eta w)\, dz\, dw
\end{align*}
By applying once again Lemma \ref{lemmaConvolution},
\begin{align}
& \int_0^t \frac{x}{t-\tau} g(t-\tau, x)\, U_2(\tau)\, d\tau \label{onceAgain2}\\
= & \eta \int_0^\infty e^{-c w}  \int_0^\infty  \, g(t, x+z+\sigma w)\, l(z, \eta w)\, dz\, dw \notag\\
= & \eta \int_0^\infty e^{-c w}  \int_{\sigma w}^\infty   \, g(t, x+z)\, l(z-\sigma w, \eta w)\, dz\, dw \notag
\end{align}

\noindent {\it iii) The solution $u(t,x)$}.\\

We get that
\begin{align*}
& u(t,0) \\
= & \int_0^\infty e^{-c w} \int_{\sigma w}^\infty g(t, z) \left(\sigma \alpha \frac{\eta w}{z-\sigma w} l(z-\sigma w, \eta w) + \eta \, l(z-\sigma w, \eta w) \right)dz\, dw
\end{align*}
and 
\begin{align*}
u(t,x) = Q^D_t \mathbf{1}_{[0, \infty)}(x) + \int_0^t \frac{x}{t-\tau} g(t-\tau, x) u(t,0)\, d\tau
\end{align*}
where, by considering \eqref{AllAtOnce} together with \eqref{onceAgain1} and \eqref{onceAgain2}, 
\begin{align*}
& \int_0^t \frac{x}{t-\tau} g(t-\tau, x) u(t,0)\, d\tau\\
= &  \int_0^\infty e^{-c w} \int_{\sigma w}^\infty g(t, x + z) \left(\sigma \alpha \frac{\eta w}{z-\sigma w} l(z-\sigma w, \eta w) + \eta \, l(z-\sigma w, \eta w) \right)dz
\end{align*}

\noindent {\it iv) The inverse process $\bar{L}_t$}.\\

Let us consider $\eta=1$ and denote by $\bar{l}$ the function
\begin{align*}
\bar{l}(t,w) = \sigma \alpha \frac{w}{t-\sigma w} l(t-\sigma w, w) + l(t-\sigma w, w), \quad t\geq \sigma w \geq 0.
\end{align*}
We notice that
\begin{align*}
 \alpha \frac{w}{t-\sigma w} l(t-\sigma w, w) = \int_0^t \delta(\tau - \sigma w)\, \alpha \frac{w}{t-\tau} l(t-\tau, w)\, d\tau
\end{align*}
where, by Lemma \ref{BertoinNOproof} and formula \eqref{symbStable},
\begin{align*}
\int_0^\infty e^{-\lambda t} \alpha \frac{w}{t} l(t, w)\, dt = \int_0^\infty e^{-\lambda t} h(w, t)\, dt = e^{-w \lambda^\alpha}.
\end{align*}
Moreover,
\begin{align*}
l(t-\sigma w, w) = \int_0^t \delta(\tau- \sigma w)\, l(t-\tau, w)\, d\tau
\end{align*}
where, by formula \eqref{Lapl}, 
\begin{align*}
\int_0^\infty e^{-\lambda t} l(t,w)\, dt = \lambda^{\alpha-1} e^{-w \lambda^\alpha}.
\end{align*}
Thus, by collecting the previous points, 
\begin{align*}
\bar{l}(t,w) = \int_0^t \delta(\tau-\sigma w) \left(\sigma \alpha \frac{w}{t-\tau} l(t-\tau, w) + l(t-\tau, w) \right)d\tau
\end{align*}
from which we get
\begin{align*}
\int_0^\infty e^{-\lambda t} \bar{l}(t,w) \, ds 
= & e^{-\sigma w \lambda} \left( \sigma e^{-w \lambda^\alpha} + \lambda^{\alpha-1} e^{-w \lambda^\alpha} \right)\\ 
= & (\sigma +\lambda^{\alpha-1}) \, e^{- \sigma w \lambda - w \lambda^\alpha}.
\end{align*}
Since $\bar{H}_w = \sigma w + H_{w}$, $w\geq 0$, in view of the relation
\begin{align*}
\mathbf{P}_0(\bar{H}_w < t) = \mathbf{P}_0(\bar{H}^{-1}_t > w) 
\end{align*}
where $ \bar{H}^{-1}_t = \inf \{w\geq 0\,:\, \bar{H}_w > t\}$, 
from the formula 
\begin{align*}
\int_0^\infty e^{-\lambda t} \mathbf{P}_0(\bar{H}_w < t)\, dt = \frac{1}{\lambda} \mathbf{E}_0[e^{-\lambda \bar{H}_w}] = \frac{1}{\lambda} e^{-\sigma w \lambda - w \lambda^\alpha}
\end{align*}
we write
\begin{align*}
- \frac{d}{dw} \int_0^\infty e^{-\lambda t} \mathbf{P}_0(\bar{H}^{-1}_t > w)\,dt 
= &  - \frac{d}{dw} \int_0^\infty e^{-\lambda t} \mathbf{P}_0 (\bar{H}_w < t)\, dt \\
= & \frac{\sigma \lambda + \lambda^\alpha}{\lambda} e^{-w \sigma \lambda - w \lambda^\alpha}\\
= & \int_0^\infty e^{-\lambda t} \bar{l}(t,w)\, dt.
\end{align*}
Thus, we conclude that
\begin{align*}
\mathbf{P}_0(\bar{H}^{-1}_t \in dw) = \bar{l}(t,w)\, dw
\end{align*}
where $\bar{H}^{-1}_t \stackrel{law}{=} \bar{L}_t$ is an inverse to the process
\begin{align*}
\bar{H}_t = \sigma t + H_t, \quad t\geq 0, \quad \eta=1.
\end{align*}
A change of variable shows that, for $\eta \geq 0$, 
\begin{align*}
\mathbf{P}_0( \eta^{-1} \bar{L}_t \in dw) = \bar{l}(t,w)\, dw
\end{align*}
where 
\begin{align*}
\bar{l}(t,w) = \sigma \alpha \frac{\eta w}{t-\sigma w} l(t-\sigma w, \eta w) + \eta \, l(t-\sigma w, \eta w), \quad t\geq \sigma w, \quad w>0
\end{align*}
and $\bar{L}_t$ is an inverse to the process
\begin{align*}
\bar{H}_t = \frac{\sigma}{\eta} t + H_t, \quad t\geq 0, \quad \eta \geq 0.
\end{align*}

\end{proof}


\end{document}